\documentclass[preprint]{elsarticle}
\usepackage[english]{babel}
\usepackage[utf8]{inputenc}
\usepackage{amsmath,amsfonts,amssymb,color,amsthm}
\usepackage{epstopdf}

\usepackage{natbib}

\usepackage{url}

\usepackage{booktabs}
\usepackage[ruled,vlined,algo2e]{algorithm2e}
\usepackage{mathdots}
\usepackage{float}
\usepackage{physics}
\usepackage{enumitem}
\newcommand{\R}{\mathbb R}
\newcommand{\N}{\mathbb N}

\def\eps{\varepsilon}

\newtheorem{theorem}{Theorem} 
\newtheorem{lemma}{Lemma}
\newdefinition{obs}{Remark}

\newtheorem{problem}{Problem}

\usepackage{multirow,bigdelim}
\usepackage[detect-none]{siunitx}
\usepackage[caption=false]{subfig}

\usepackage{booktabs}
\usepackage{xcolor}

\usepackage{float}
\usepackage{mathtools}
\newcommand{\ve}{\textbf{vect}}

\newcommand{\pt}{\tilde{p}}
\DeclareMathOperator*{\argmin}{arg\,min}

\graphicspath{{./figures/}}

\begin{document}
	\title{A gradient system approach for\\ Hankel structured low-rank approximation}
	
	\author[1,2]{Antonio Fazzi\corref{cor1}}
	\ead{antonio.fazzi@gssi.it}
	\author[1]{Nicola Guglielmi} \ead{nicola.guglielmi@gssi.it}
	\author[2]{Ivan Markovsky} \ead{imarkovs@vub.ac.be}
	\cortext[cor1]{Corresponding author}
	\address[1]{Gran Sasso Science Institute (GSSI), Viale F. Crispi 7, 67100 L'Aquila, Italy}
	\address[2]{Vrije Universiteit Brussel (VUB), Department ELEC, Pleinlaan 2, 1050 Brussels, Belgium}

\begin{abstract}
Rank deficient Hankel matrices are at the core of  several applications. 
However, in practice, the coefficients of these matrices are noisy due to e.g. measurements 
errors and computational errors, so generically the involved matrices are full rank. 
This motivates the problem of Hankel structured low-rank approximation. 
Structured low-rank approximation problems, in general, do not have a global and efficient  
solution technique. In this paper we propose a local optimization approach based on a two-levels 
iteration. Experimental results show that the proposed algorithm usually achieves good accuracy 
and shows a higher robustness with respect to the initial approximation, compared to alternative 
approaches. 
\end{abstract}


\maketitle

\section{Introduction}
\subsection{Preliminaries}
A Hankel matrix $H \in \R^{m \times n}$ is a structured matrix 
%of the form
%$$
%H = (h_{i+j})_{i, j=1}^{m, n}
%$$
where the entry  on the $i$-th row and the $j$-th column depends only on the sum $i+j$. 
Hankel matrices can be associated in a natural way to vectors or time series.
For a given $m \in \N$, let $H_m \in \R^{m \times n}$ ($m \leq n$)  be the Hankel matrix built from the 
real numbers $p_1, p_2, \dots, p_T$ (with $T=n+m-1$) as follows
$$
H_m(p)=\begin{pmatrix}
p_1 & p_2 & \cdots & p_n \\
p_2 & p_3 & \cdots & p_{n+1} \\
\vdots & \vdots &   & \vdots \\
p_m & p_{m+1} & \cdots & p_T
\end{pmatrix}
$$

We denote by $\mathcal{H}^{m\times n}$ the subspace of $m \times n$ Hankel matrices.
Given a vector
\[p = (p_1, \dots, p_T)^{\top},\]
we define the Hankel matrix 
%$H_m(p) \in \R^{m \times \bf{\cdot}} \subset \mathcal{H}$ 
\[H=H_m(p) \in \mathcal{H}^{m \times (T-m+1)}.\]
Conversely, given a Hankel matrix
\[H \in \mathcal{H}^{m \times (T-m+1)},\]
we define the vector
\[\ve (H) =p = (p_1, \dots, p_T)^{\top}.\]

Hankel structured matrices arise in applications in different  areas of sciences, such as control theory, approximation and interpolation  problems, 
polynomials computations (see for example \cite{BiniPan,Peller2003}). Most applications involve the computation of the distance between a 
given Hankel matrix and a singular one (preserving the same structure).  
In the following we focus our attention on system identification of linear time-invariant models and polygons from moments reconstruction.  The 
structure preserving property is of interest both theoretically (because of the applications involving singular Hankel matrices) \cite{ParkLeiRosen, 
ShawPokalaKumaresan, ChuFunderlicPlemmons} and in practice (due to the presence of noise in real life problems) \cite{MarkovskySLRA}.

In this paper,  we analyze the problem of computing the structured distance to singularity in the case of  (scalar) Hankel matrices. Because of the association previously discussed between a Hankel matrix and the vector of its coefficients, we consider the following   formulation of the  problem 
\begin{equation}
\label{prob:problem1}
\min_{\substack{\pt \in \R^T \\   H_m(\pt)  \  \text{rank deficient}}} 
 \| p - \pt \|_w
\end{equation}
where the norm 
\begin{equation*}
\| p \|_w = \left( \sum_{i=1}^{T} w_i |p_i|^2 \right)^{1/2}. 
\end{equation*}
For example,  the weights can be chosen such that $\| p \|_w = \| H_m(p) \|_F$. 

%\begin{obs}
%	The norm in \eqref{prob:problem1} weights the elements in the vector $p$ so that the norm of the vector is the same as the Froenius norm of the associated Hankel matrix. Because of the association between a Hankel matrix and the vector of its coefficients, we will switch from vector norms to matrix norms by considering the appropriate metric. 
%\end{obs}

\subsection{Main contribution of the paper}
The  approximation of a structured matrix by a matrix of lower rank which preserves the same structure is a classical problem in numerical linear algebra \cite{ChuFunderlicPlemmons,MarkovskySLRA}.  The constraint on the structure of the computed solution makes the problem harder in comparison with  unstructured low-rank approximation. Currently, there is no analytical solution neither standard solution strategies.
% In this paper, we focus our attention on the case of Hankel structured matrices.  
Looking at numerical schemes,  global optimization approaches are usually computationally expensive, hence the most common methods for solving  \eqref{prob:problem1} are local optimization  \cite{Markovskyvarproj} or convex relaxation methods \cite{Fazelthesis}.

 The algorithm we  propose here is based on the ideas presented in \cite{Fazzi2019}, where the approximate common factor problem is restated as a structured low-rank approximation problem for Sylvester matrices (block Toeplitz matrices).  
 On the other hand we take inspiration from the ideas presented in \cite{ButtaGuglielmiNoschese}, where the authors study the behavior of the eigenvalues of Toeplitz matrices under finite (structure preserving) perturbations. 
 The  results in \cite{Fazzi2019, ButtaGuglielmiNoschese} motivate the work presented in this paper because of the similarities between Toeplitz and Hankel matrices. 

 We  propose a double iteration  method for Hankel structured low-rank approximation. 
 According to the formulation in  \eqref{prob:problem1},  we start from a data vector $p$ (whose associated Hankel matrix is full rank) and  we aim at modifying 
the matrix $H_m(p)$ in order to make it singular, by adding a perturbation of the form $\epsilon \delta$ (where $\delta$ is a norm  $1$ vector, while $\epsilon$ is a scalar  measuring the norm of the perturbation).  
The two values of $\epsilon$ and $\delta$ are updated on two different independent levels:
 \begin{itemize}
 \item at the inner level we fix the norm of the perturbation $\epsilon$ and we look for the vector $\delta$  which minimizes the smallest singular value of  
 $H_m(p +   \epsilon \delta)$.   
 This is done by an ODE for $\delta$ which is given by the gradient system for the smallest singular value of the matrix;
 \item at the outer level we need to move the value of $\epsilon$  (increase the norm of the perturbation)  till the smallest singular value is zero. 
 \end{itemize}
 
 \paragraph*{Outline}
 The goal of the paper is to provide a numerical algorithm for the solution of  \eqref{prob:problem1} and to test it on the  applications of system identification and polygons from moments reconstruction in order to observe its performances.
 The application to identification of linear time-invariant dynamical models is presented in Section~\ref{sec:sysid} and the polygons from moments reconstruction is presented in Section~\ref{sec:polfrommoments}.
%  In Section \ref{sec:appl} we first focus on some motivating applications of interest for the considered problem: system identification of linear time-invariant dynamical models (Section \ref{sec:sysid}) and polygons from moments reconstruction (Section \ref{sec:polfrommoments}).
  In Section~\ref{sec:alg} we present the algorithm, describing in  details  the two levels iteration, and showing the monotonicity of the smallest singular value at the inner level. 
	We perform then some numerical tests in Section~\ref{sec:numtest} doing a comparison (whenever it is possible) with the SLRA package~\cite{UsevichMarkovsky2}.  
	A summary  and possible directions for future work are listed in  Section~\ref{sec:concl}.

\section{Applications}
\label{sec:appl}
We describe in this section some applications involving low-rank approximation of Hankel matrices.
\subsection{System Identification}
\label{sec:sysid}
Hankel matrices play a central role in system theory and identification. An important result states that the rank deficiency of a Hankel matrix built from some data 
$p = (p(1), p(2), \dots, p(T)) \in \R^T$  is equivalent with the fact that such data $p$ is an impulse response of a  linear time-invariant dynamical system of order equal to the rank of the matrix.
In particular it is known that a linear time-invariant dynamical system of  order $m$ can be defined through a difference equation  
 \cite[Theorem~7.2]{Markovskybook2006} of the form
\begin{equation}
\label{eq:diffeq}
R_0 p(t) + R_1 p(t+1) + \cdots + R_m p(t+m) = 0  \ \ \text{for all} \ t. 
 \end{equation}
where $R = (R_0, R_1, \dots, R_m) \neq 0$ is a vector of real numbers.
In the problem of system identification we are given a finite trajectory  of the system  $p = (p(1), p(2), \dots, p(T)) \in \R^T$ and the order $m$ and we aim to find   the generating model (defined by the vector of parameters $R$).
If we write equation \eqref{eq:diffeq} in matrix form we have
\begin{equation}
\label{model}
(R_0, R_1, \dots, R_m) 
%\begin{pmatrix}
%p(1) & p(2) & \cdots & p(T - m) \\
%p(2) & p(3) & \dots & p(T-m+1) \\
%\vdots & \vdots &  & \vdots \\
%p(m+1) & p(m+2) &  & p(T)
%\end{pmatrix}  
H_{m+1}(p)
= 0,
\end{equation}
which means the Hankel matrix $H_{m+1}(p)$ is rank deficient. 

 We deduce that a necessary and sufficient condition for the time series $p$ to be generated by a linear time-invariant dynamical model is $H_{m+1}(p)$ to be rank deficient. However in practical applications  $p$ can be corrupted by noise, so we deal with a full rank Hankel matrix, and we are interested in computing the closest rank deficient Hankel matrix. The vector $R$ is then computed as the kernel of the approximating  matrix. 
 
Model \eqref{model} is used in several contexts, ranging from biomedical signal processing, vibrational analysis of dynamical structure, industrial process system identification, stochastic identification and telecommunications \cite{VanhuffelLemmerling, BultheelDeMoor, Demoor1993, Demoor1994, Demoor1997}.
 
 \subsection{Polygons from moments reconstruction}
 \label{sec:polfrommoments}
 The {\em polygons from moments problem} consists in reconstructing a (binary) simply connected and
 nondegenerate polygon from some complex quantities called
 \textit{moments}. The mathematical framework of the problem and its
 derivation can be found, e.g., in
 \cite{MilanfarVerghese,MilanfarPutinar}; we briefly summarize here how
 to restate the problem as a Hankel low-rank approximation problem.
 
 A polygon is reconstructed from  its $n$ vertices $z_1, \dots, z_n$, and each~$z_i$ is a complex number. For a given integer  number $N > 2n$, the so called \textit{complex moments} are defined as \cite{MilanfarVerghese}
 \begin{equation}
 \label{eq:complexmoments}
 \tau_k = \sum_{j=1}^n  a_j z_j^k  \ \ \ for \ k=0, \dots N
 \end{equation}
 where 
 \begin{equation*}
 a_j = \frac{2 A_j}{(z_j - z_{j-1})(z_j - z_{j+1})}, \ \ A_j = \frac{i}{4} \det \begin{pmatrix}
 z_{j-1} & \bar{z}_{j-1} & 1 \\ 
 z_j & \bar{z}_j & 1 \\
 z_{j+1} & \bar{z}_{j+1} & 1
 \end{pmatrix}
 \end{equation*}
 and we assume that the set of vertices is cyclic (so $z_0 = z_n, z_{n+1} = z_1$ and so on). 
 Equation \eqref{eq:complexmoments} can be written in matrix  form  as
 \begin{equation}
 \label{eq:complexmomvector}
 \begin{pmatrix}
 \tau_0 \\ \tau_1 \\ \vdots \\ \tau_N
 \end{pmatrix} = \begin{pmatrix}
 1  & 1 & \cdots & 1 \\ z_1 & z_2 & \cdots & z_n \\ \vdots & \vdots & \ddots & \vdots \\ z_1^N & z_2^N & \cdots & z_n^N
 \end{pmatrix} \begin{pmatrix}
 a_1 \\ a_2 \\ \vdots \\ a_n
 \end{pmatrix} \Longleftrightarrow  \tau = Z a
 \end{equation}
 Using Prony's method \cite{Hildebrand} it is possible to show that the set of vertices can be computed from the vector of complex moments $\tau$ in \eqref{eq:complexmomvector}. Define the polynomial 
 \begin{equation*}
 P(z) = \prod_{j=1}^n (z - z_j) = z^n + \sum_{j=1}^{n} p_j z^{n-j} 
 \end{equation*}
 and the vector $p=(p_n, p_{n-1}, \dots, p_1)$. In this way, the problem is equivalent to the computation of the vector $p$.
 Premultiplying equation~\eqref{eq:complexmomvector} by
 the following Toeplitz matrix
 \begin{equation*}
 K_{N+1}=\begin{pmatrix}
 p_n & p_{n-1} & \cdots &p_1 & 1 & &  \\
 & \ddots & \ddots & & \ddots & \ddots &  \\
 &  & p_n & p_{n-1} & \cdots &p_1 & 1
 \end{pmatrix}:
 \end{equation*}
 we get
 \begin{equation*}
 K_{N+1} \tau = K_{N+1} Z a = 0, 
 \end{equation*}
 where the last equality comes from  the definition of $P(z)$. 
 The identity $K_{N+1} \tau = 0$ can be then written as
 \begin{equation}
 \label{eq:polygonhlra}
 \begin{pmatrix}
 \tau_0 & \tau_1 & \cdots& \tau_{n} \\
 \tau_1 & \tau_2 & \cdots & \tau_{n+1} \\
 \vdots & \vdots &      & \vdots \\
 \tau_{N-n} & \tau_{N-n+1} & \cdots & \tau_N
 \end{pmatrix} \begin{pmatrix}
 p \\ 1
 \end{pmatrix} = 0. 
 \end{equation}
 Equation \eqref{eq:polygonhlra} shows that the Hankel matrix of the moments is rank deficient. 
 In realistic applications the measurements of the complex moments is affected by noise, so that 
 we expect the Hankel matrix $H_{N-n}(\tau)$ to be full rank. 
 By computing the closest rank deficient  Hankel matrix and its kernel, we can approximately 
 reconstruct the set  of vertices $z_1, \dots, z_n$.
 
 \begin{obs}
 In the application to polygons from moments reconstruction, the data are complex-valued.
 \end{obs}

\section{The algorithm}
\label{sec:alg}
In this section, we propose  an algorithm for the numerical solution of scalar Hankel structured low-rank approximation problem \eqref{prob:problem1}. 
 
 The optimization problem we aim at solving is written as follows.

 \begin{problem}
 \label{pb:Hankelmin}
 Given a vector $p \in \R^T$, the  Hankel matrix $H_m(p) \in \R^{m \times n}, m \leq n$ and the norm $\| \cdot \|_w$, compute
 \begin{equation}
 \label{eq:probFnorm}
 \min_{\tilde{p} \in \R^T}  \| p - \tilde{p} \|_w  \ \ \ \text{subject to} \ \ \  rank(H_m(\pt)) < m. 
 \end{equation}
 \end{problem}
The objective function we minimize is the Frobenius norm of the difference of the initial Hankel matrix $H_m(p)$ and the approximating rank deficient Hankel matrix $H_m(\pt)$.  
 We illustrate later how to generalize the proposed approach to different weighted norms. 
 %The algorithm we  propose computes such a solution splitting the problem on two different levels: at the inner level we fix the norm of the perturbation on the starting vector and we look for a direction which makes the smallest singular values of the perturbed Hankel matrix decreasing; at the outer level we increase the norm of the total perturbation till we get an admissible solution (a vector such that the associated Hankel matrix is rank deficient). The two (independent) updates of these quantities lead to the proposed local optimization approach. 

% \begin{enumerate}
% \item at the inner level we fix the norm of the perturbation on the starting matrix, and we compute an optimal (structured) perturbation of (Frobenius) norm $1$ which minimizes the smallest singular value of the perturbed matrix; this is done by following the steepest descent direction for the singular value of interest;
% \item at the outer level we move the norm of the perturbation on the starting matrix till we reach an admissible solution (a rank deficient Hankel matrix).
% \end{enumerate}
 
 \subsection{A double iteration algorithm}
 
Let $\sigma_{\min}(A)$ be the smallest singular value of the matrix $A$.
 It is well known that the rank of a matrix can be computed looking at its singular values. 
 For a matrix $H$, the value $\sigma_{\min}(H)$ measures its (unstructured) distance to singularity. Our aim here is to iteratively decrease the value of $\sigma_{\min}(H)$ but preserving the Hankel structure. 
This is done by modifying (in a structured way) the starting Hankel matrix $H$  in a way that makes the functional $\sigma_{\min}(H)$ decreasing till it reaches a fixed (small) tolerance.  In such a way we are able to achieve (to the given tolerance) the rank constraint and preserve the structure at the same time.
  
 In particular, starting from a (full rank) Hankel matrix $H_m(p)$, the perturbed (Hankel) matrix has the form $H_m(p + \epsilon \delta)$, with $\delta$ a  norm $1$ vector   and  $\epsilon \in \R$ a number which measure the norm of the perturbation on the starting vector $p$; the values of the two parameters $\delta$ and $\epsilon$ are updated independently on two different levels: 
 \begin{itemize}
 \item for a fixed value of $\epsilon$ we compute an optimal perturbation $\delta$ (by looking at the stationary points of a suitable gradient system of ordinary differential equations);
 \item once we have computed the perturbation $\delta$, which we maintain fixed, we update the value of $\epsilon$ in order to get closer to an admissible  solution (a vector associated with a rank deficient Hankel matrix).
 \end{itemize}
% The goal of the following sections is to describe in  details how to build such a perturbation. 
 
 \subsection{How to decrease the smallest singular value}
 In this section we consider $\epsilon$ as fixed and we want to compute a perturbation vector   $\delta$ (of  norm 1)  in such a way that $\sigma_{\min}(H_m(p + \epsilon \delta))$ is minimum among all the possible norm-$1$ perturbations. 
 In other words we are looking for the norm $1$ vector $\delta$ which makes $\sigma_{\min}(H_m(p + \epsilon \delta))$ decreasing along the steepest descent direction. 
We make use of the following standard result about perturbation of eigenvalues for positive semidefinite matrices \cite{Kato}, which is adapted to the case of singular values
  \begin{lemma}
\label{lemma:deriveig}
	Let $D(t)$ be a differentiable matrix-valued function for $t$ in a neighborhood of $t_0=0$.
%and let $\lambda(t)$ be an eigenvalue of $C(t)$ 
%converging to a simple eigenvalue $\lambda_0$ of $C(0)$ as $t \rightarrow 0$.   Let $x_0$ be a normalized eigenvector of $C(0)$ associated with $\lambda_0$. 
%Then the function $\lambda(t)$ is differentiable near $t=0$ with
%$$
%\dot{\lambda} = x_0^{\top} \dot{C} x_0.
%$$
Let $D(t) = U(t) \Sigma(t) V(t)^{\top}$ be a smooth (with respect to $t$) singular value decomposition of the matrix $D(t)$ and
$\sigma(t)$ be a certain singular value of $D(t)$ converging to a simple singular value  $\sigma_0$ of $D(0)$. 
If $u_0, v_0$ are the associated left and right singular vectors, respectively,
the function $\sigma(t)$  is differentiable near $t=0$ with 
\begin{equation*}
\dot{\sigma} = u_0^{\top} \dot{D} v_0.
\end{equation*}
\end{lemma}
Lemma \ref{lemma:deriveig} is  applied to the singular values of the matrix $H_m(p + \epsilon \delta)$. 
Remembering that both $p$ and $\epsilon$ are constant, the derivative of $\sigma_{\min}(H_m(\pt))$ is given by
\begin{equation}
\label{eq:derivsigma}
\dot{\sigma} =  u^{\top} H_m(\epsilon \dot{\delta}) v,
\end{equation}
where % $\epsilon$ is the scalar which satisfy the relation $H_m(\tilde{p}) = H_m(p) + \bar{\epsilon} H_m(\delta)$, and 
$u, v$ are the left and right singular vectors associated with $\sigma_{\min}(H_m(\pt))$.

For a pair of real matrices $A, B$ we let
\[
\langle A, B \rangle = {\rm trace} (A^\top B) = \sum_{i,j=1}^n  A_{ij} B_{ij}
\]
the Frobenius inner product, inducing the norm $\| A \| = \langle A, A \rangle^{1/2}$. If instead $A$ and $B$ are vectors, $\langle A, B \rangle$
denotes the standard inner product.
 
The optimal descent direction for the singular value of interest is obtained by minimizing the following function % (up to constant terms)
\begin{equation}
\label{eq:functiontomin}
 u^{\top} H_m(\dot{\delta}) v = \langle u v^{\top},  H_m(\dot{\delta}) \rangle = \langle P_{\mathcal{H}} (u v^{\top} ), H_m(\dot{\delta}) \rangle
\end{equation}
where $P_{\mathcal{H}}(B)$ denotes the orthogonal projection of the matrix $B$ onto the subspace of Hankel matrices $\mathcal{H} = \mathcal{H}^{m \times n}$. 
% the last equality is a consequence of the Hankel structure of the matrix $H_m(\dot{\delta})$. 

Since we may work directly on the vectors associated to the Hankel matrices, the objective function in \eqref{eq:functiontomin}  
can be written as 
\begin{equation}
\label{eq:vectorialfuntomin}
\langle \ve (P_{\mathcal{H}} (u v^{\top} )), \dot{\delta} \rangle
\end{equation} 
The explicit formula for the operator $P_{\mathcal{H}}(\cdot)$ is given in the following Lemma \ref{lemma:formulaproj}
\begin{lemma}
\label{lemma:formulaproj}
Let $\mathcal{H}^{m \times n}$ be the linear manifold of real $m \times n$ Hankel matrices, and let $B \in  \R^{m \times n}$ an arbitrary matrix. 
The orthogonal projection (with respect to the Frobenius inner product) of $B$ onto $\mathcal{H}$ is given by 
\[
P_{\mathcal{H}}(B) = H_m(q)
\]
where
$$
\begin{aligned}
q_i &=  \left\{ \begin{array}{l} \displaystyle{\frac{1}{i} \sum_{j=1}^i B_{j,i-j+1}   \ \ \ \text{for} \ \  i=1, \dots, m} \\[4mm]
                                 \displaystyle{\frac{1}{m}  \sum_{j=1}^{m}     B_{j, i-j+1}         \ \ \ \text{for} \ \  i=m+1, \dots, n-1}      \end{array} \right.
																\\
q_{m+n-i} &= \frac{1}{i} \sum_{j=1}^{i} B_{j,m+n-i-j+1}   \ \ \ \text{for} \ \  i=1, \dots, m.      
\end{aligned}
$$
\end{lemma}
\begin{proof}
The solution is given by the orthogonal projection with respect to the Frobenius inner product which is simply obtained by taking the averages along the anti-diagonals
of the matrix $B$. 
\end{proof}
A further important result follows.
% is related to the projection $P_{\mathcal{H}}$, and in particular it holds true that $P_{\mathcal{H}}(uv^{\top}) \neq 0$, where $u$ and $v$ are the singular vectors associated with $\sigma(H_m(\pt))$.  
% In this way the zeros of the function in \eqref{eq:functiontomin} depend only on $\dot{\delta}$. 
\begin{lemma}
\label{lemma:projnonzero}
Let $H(p)$ be a Hankel matrix.  If $\sigma > 0$ is a simple singular value of $H(p)$ and $u$ and $v$ are the corresponding left and right singular vectors, then
$$
P_{\mathcal{H}}(uv^{\top}) \neq 0.
$$
\end{lemma}
\begin{proof}
Assume, by contradiction, $P_{\mathcal{H}}(uv^{\top}) = 0$. 
We have 
$$
0 = \langle P_{\mathcal{H}}(uv^{\top}) , H(p) \rangle = \langle u v^{\top}, H(p) \rangle = u^{\top} H(p) v = \sigma.
$$
The proof is completed since $\sigma > 0$ by assumption, so by following the chain we get the contradiction $0 > 0$.
\end{proof}

Consider the singular value $\sigma_{\min}(H_m(\pt))$, and let  $u, v$ be the corresponding left and right singular vectors, respectively. 
The steepest descent direction for the considered singular value is given by the solution of the following optimization problem:
\begin{equation}
\label{eq:optprob}
\dot{\delta}^* = \argmin_{\substack{\dot{\delta} \in \R^T \\ \| \dot{\delta} \|_2 = 1}} u^{\top} H_m(\dot{\delta}) v  \ \ \  \ \ 
\text{subject to}  \  \langle \delta, \dot{\delta} \rangle = 0
\end{equation}
where the constraint on the norm guarantees the uniqueness of the solution (which represents a direction) while the last constraint guarantees 
the norm conservation of $\delta$, which we have assumed. We give the solution of the problem \eqref{eq:optprob} in the following Lemma
\begin{lemma}
\label{lemma:ode}
Let   $u, v$ be the left and right singular vectors of   $H_m(\pt) = H_m(p + \epsilon \delta)$ with $\delta \in \R^T$ of unit norm.  The solution of the optimization problem \eqref{eq:optprob} is given by 
 \begin{equation}
 \label{eq:opt}
 \mu \dot{\delta}^* = \ve(-P_{\mathcal{H}} (uv^{\top})) + \langle \delta, \ve(P_{\mathcal{H}} (uv^{\top})) \rangle \delta
 \end{equation}
  where $\mu$ is the norm of the vector in the right hand side.
\end{lemma}

\begin{proof}
In the Frobenius metric,  the minimizing direction for  function   \eqref{eq:functiontomin} is  reached  for    $H_m(\dot{\delta}) = -P_{\mathcal{H}} (uv^{\top})$, consequently  $\ve(-P_{\mathcal{H}} (uv^{\top}))$ is the solution of the unconstrained version of \eqref{eq:optprob}.

For an arbitrary vector $y$ (of suitable dimension), the projection onto the space orthogonal to $\delta$ is given by 
$$
P(y) = y - \langle \delta, y \rangle \delta.
$$
The claim follows by choosing $y = \ve(-P_{\mathcal{H}} (uv^{\top}))$ and normalizing in order to satisfy the constraint on the norm. 
\end{proof}
Lemma \ref{lemma:ode} is a key result for the proposed method.
 Since \eqref{eq:opt} gives the unit norm steepest descent direction $\delta$
for the smallest singular value, we may omit the scaling factor $\mu$ by normalizing $\dot{\delta}$ (to have norm $1$)  and
consider the gradient system for $\sigma$
%Consequently it is natural to consider the gradient system for $\sigma$,
\begin{equation} \label{eq:ode}
\begin{aligned}
\dot{\delta} &= -\ve(P_{\mathcal{H}} (uv^{\top})) + \langle \delta, \ve(P_{\mathcal{H}} (uv^{\top})) \rangle \delta \\
 \delta(0) &= \delta_0 \  \text{
with} \ \ \  \| \delta_0 \| = 1
\end{aligned}
\end{equation}
that (locally) minimizes the smallest singular value $\sigma$ of $H_m(p + \epsilon \delta)$ on the set $\{ \delta : \| \delta \| = 1 \}$.

Considering the initial value $\delta(0)$ of unit norm then we can see that
\begin{equation*}
\frac{d}{dt} \| \delta(t) \|^2 =
2\,\langle \dot{\delta}(t), \delta(t) \rangle =
0
\end{equation*}
where we have replaced $\dot{\delta}$ by the right-hand side of the
ODE \eqref{eq:ode}, which implies norm conservation (this relation holds true since we impose $\delta(t)$ to have norm $1$). 

By construction \eqref{eq:ode} is the gradient system for the objective functional $\sigma$ (the smallest singular value of the perturbed Hankel matrix) 
under the constraint that the perturbation has fixed norm $\eps$ (i.e. $\delta$ has unit  norm).

Thus it follows directly that the objective functional  is monotonically decreasing along the solution of \eqref{eq:ode}.

\begin{theorem}
\label{th:monotonicity}
Let $\delta(t) \in \R^T$  be a solution of \eqref{eq:ode}. If $\sigma(t)$ is the smallest singular value of the Hankel matrix $H_m(\pt) = H_m(p + \epsilon \delta)$, then 
$$
\dot{\sigma}(t)  \leq 0.
$$
\end{theorem}
\begin{proof}
To prove the result we recall that  $\dot{\sigma} = u^{\top} H_m(\dot{\delta}) v$ (omitting the constant term) and we compute its point of minimum 
by using the expression in \eqref{eq:vectorialfuntomin}. We replace now the expression of $\dot{\delta}$  using equation~\eqref{eq:ode}, and observe that
\begin{eqnarray*}
&& \langle \ve (P_{\mathcal{H}} (u v^{\top})), \ve (P_{\mathcal{H}} (u v^{\top})) \rangle =  \| \ve ( P_{\mathcal{H}} (u v^{\top})) \|^2.
\\
&& \langle \ve (P_{\mathcal{H}} (u v^{\top})),    \delta \rangle  \langle  \ve (P_{\mathcal{H}} (u v^{\top})), \delta \rangle  = \langle \ve (P_{\mathcal{H}} (u v^{\top})),    \delta \rangle^2.
\end{eqnarray*}
By adding the two terms with the correct signs, we get
\[
\langle  \ve (P_{\mathcal{H}} (u v^{\top})),    \dot{\delta} \rangle = -  \| \ve ( P_{\mathcal{H}} (u v^{\top})) \|^2 + \langle \ve (P_{\mathcal{H}} (u v^{\top})),    \delta \rangle^2 \leq 0
\]
using Cauchy-Schwarz inequality and recalling that  $\delta$ has norm $1$. 
\end{proof}

	\begin{obs}
		Tha assumption on $\delta$ to have norm $1$ plays a key role. However we remind that the whole perturbation is given by $\epsilon \delta$ whose norm is $\epsilon$ (fixed at the inner level). The scaling factor for the direction vector $\delta$ is arbitrary, so we  choosed it to be norm $1$ without loss of generality. 
		\end{obs}
	
	 The previous results allow to state that the points of (local) minimum for the objective functional ($\delta \rightarrow \sigma_{\min}(H_m(p + \epsilon \delta))$)  are the stationary points of the gradient system \eqref{eq:ode} ($\delta$ such that $\dot{\delta} = 0$). 
%This is because the singular value to be minimized is monotonically decreasing along the solutions of the system of differential equations \eqref{eq:ode}, therefore the zeros of its derivative are the candidate solutions. 
Moreover Lemma \ref{lemma:projnonzero} guarantees that such points of minimum correspond only to the zeros of the derivative and vice versa. Consequently, the goal of the inner iteration is to integrate equation~\eqref{eq:ode} until stationary points.

Equation \eqref{eq:ode} is a differential equation for the vector $\delta$. In order to compute its stationary points we choose to adopt an Explicit Euler scheme, because the function evaluation (a singular value decomposition) at each step is quite expensive.  Because of the significant computational cost we  preferred to avoid both higher order explicit schemes and implicit method (the first require more than one svd factorization at each step, the latter require an extra  solution of a system).
The numerical scheme for the integration of the ODE \eqref{eq:ode} is summarized in Algorithm \ref{alg:controlODE}.

\begin{algorithm2e}[H]
\caption{Numerical solution of the ODE \eqref{eq:ode}}
\label{alg:controlODE}
\DontPrintSemicolon
\KwData{$\delta$, $\sigma$, $u$, $v$, $h$ (step Euler method) and $\gamma$ (step size reduction),~$\epsilon$.}
\KwResult{$\delta^*$, $\sigma^*$}

\Begin{
\nl Compute $\dot{\delta}= \ve(- P_{\mathcal{H}}(uv^{\top})) + \langle \delta, \ve(P_{\mathcal{H}}(uv^{\top})) \rangle \delta $ \;
\nl Euler step $\rightarrow$ $ \delta^{1} = \delta + h \dot{\delta}$ \;
\nl Normalize $\delta^{1}$ dividing it by its  norm \;
\nl Compute the  singular value $\sigma^1=\sigma_{\min}(H_m(p + \epsilon \delta^1))$ \;
\nl Compute the  singular vectors $u^1$ and $v^1$  associated with $\sigma^1$ \;
\nl \eIf{ $\sigma^1 > \sigma$}{
      reject the result and reduce the step $h$ by a factor $\gamma$ \;
      repeat from line 2} 
      {accept the result; set  $\sigma^* = \sigma^1$, $\delta^*=\delta^1$\;
 \nl \If{$\sigma^* - \sigma < tol$ \textbf{or} $\sigma^* \leq tol$}
 {\textbf{return}} \;
    } 
% \nl \eIf{$h^{1} = h$}
% {increase the step size of $\gamma$, $\bar{h} = \gamma h$}
% {set $\bar{h} = h$} \;
 %\nl Go to the next iteration \;
 }
\end{algorithm2e}
The singular triplet $u, v, \sigma$ at each step is computed through the Matlab function \textit{svds} in order to store only the needed vectors and not all the factorization of the Hankel matrix. However the numerical results got from \textit{svds} and \textit{svd} are not exactly the same in floating point arithmetic, so the use of \textit{svds} can lead to a less accurate solution because of algorithmic errors but  the code is expected to run faster.

\subsection{How to compute $\epsilon$?  }

The goal is that of modifying $\epsilon$ so that the branch of smallest singular values $\sigma(\epsilon)$ of $H_m(p +   \epsilon \delta(\epsilon))$ 
(where $\delta(\epsilon)$ stands for the minimizer at a given $\epsilon$) smoothly reaches the minimum value zero.

\subsubsection*{Free and constrained dynamics}

Once we get the optimal perturbation $\delta(\epsilon)$ (for a given value of $\epsilon$) we need to iteratively update the value of $\epsilon$ till we reach the sought solution (a vector $\pt = p + \epsilon^* \delta(\epsilon^*)$ associated with a rank deficient Hankel matrix $H(\pt)$). A possible way to proceed would be to increase the value of $\epsilon$ by a constant increment at each step; however,  if we proceed in this way on some randomly chosen test problem, we may observe  something unexpected, % (Figure \ref{fig:strange}): 
that is an apparent loss of monotonicity of the objective functional   $\epsilon \rightarrow \sigma_{\min}(H_m(p+\epsilon\delta))$. 
%\begin{figure}
%\begin{center}
%\includegraphics[width=8cm, height=6cm]{wrongdist1.eps}
%\caption{Plot of the smallest singular value of a perturbed Hankel matrix as function of $\epsilon$ \label{fig:strange}}
%\end{center}
%\end{figure}

The  observed behavior may be explained in terms of the choice of the initial datum for the ODE \eqref{eq:ode}. 
Consider the value $\epsilon_1 \in \R$, and integrate the equation~\eqref{eq:ode} till converging to an optimal perturbation $\delta_1$ (of  norm $1$) which corresponds to a 
(locally) minimal  singular value $\sigma(\epsilon_1) = \sigma_{\min}(H_m(p +  \epsilon_1 \delta_1(\epsilon_1))$. 
Passing to the next iteration we update $\epsilon_2=\epsilon_1 + \Delta \epsilon$ where $\Delta \epsilon$ is a constant term, and a natural choice for the initial datum for the ODE \eqref{eq:ode} is $p + \epsilon_2 \delta_1(\epsilon_1)$. 
In general we may have $\sigma_{\min}(H_m(p + \epsilon_2 \delta_1(\epsilon_1))) > \sigma(\epsilon_1)$, and this could persist also once we get the optimal perturbation 
$\delta_2(\epsilon_2)$ corresponding to $\epsilon_2$ (i.e.,  it can happen that $\sigma(\epsilon_2) > \sigma(\epsilon_1)$, which means we are not following the smallest 
singular value in a smooth way). 

%%%%%%%%%%%%%%%%%%%%%%%%%%%%%%%%%%%%%%%%%%%%%%%%%%%%%%%%%%%%%%%%%%%%%%%%%%%%%%%%%%%%%%%%%%%%%%%%%%%%%%%%%%%%%%%%%%%%%%%%%%%%%%%% STOP HERE

To deal with this issue we need to introduce an intermediate different ODE, which makes the branch $\sigma(\epsilon)$ continuous. Since this is obtained by
omitting the norm constraint on $\delta$ at the inner level we call it {\em free dynamics}.
\paragraph{Free dynamics for $\sigma$}
The basic idea of this computational strategy is to start each iteration exactly from the endpoint of the previous one: we iteratively alternate two different dynamics in order to preserve the monotonicity of the function $\sigma_{\min}(H_m(p+\epsilon \delta))$ with respect to $\epsilon$ (the monotonicity with respect to $\delta$ is a property of the gradient system).

%first minimize the value of the objective function $\delta \rightarrow \sigma_{\min}(H_m(p+\epsilon \delta))$  by applying the free dynamic, i.e. by increasing the norm of the perturbation from $\epsilon_1$ to $\epsilon_2$, then we optimize    on the level set of perturbations whose norm is $\epsilon_2$    by solving \eqref{eq:ode}. 

Previously we considered \eqref{eq:ode}, which is an ODE whose solution is a vector $\delta(t)$ of unit norm, so that the norm of the perturbation on the starting data vector 
$p$ is given by the value of $\epsilon$. What happens if we remove the constraints on the computed solution? The unconstrained optimization problem is still a gradient system for the smallest  singular value $\sigma$, although not preserving $\| \delta(t) \| = 1$, and the modified ODE is given by 
\begin{equation} \label{eq:odefree}
\dot{\delta} = {}-\ve(P_{\mathcal{H}} (uv^{\top})) 
\end{equation}
Assume we have solved \eqref{eq:ode} for $\epsilon = \epsilon_1$ and in the outer iteration a value $\epsilon=\epsilon_2 > \epsilon_1$ is proposed.
The initial datum for $\delta$, that we choose, is given by $\delta(\epsilon_1)$, which has unit norm and has been computed by solving \eqref{eq:ode}. 
At some time $\bar{t}$ the solution of \eqref{eq:odefree}---which is expected to increase in norm in order to decrease the smallest singular value---has norm
\[
\delta(\bar{t}) = \frac{\epsilon_2}{\epsilon_1}.
\]
At this point we consider again \eqref{eq:ode} with $\epsilon=\epsilon_2$ and initial datum $\delta(0) = \delta(\bar{t})/\| \delta(\bar{t}) \|$.  
In this way we have a global continuity with respect to $t$ and $\epsilon$ of $\delta$ and consequently of $\sigma$, and avoid jumps on different
branches which may determine a loss of the monotonicity property that we expect.
 If we assume that 
	$\eps^*$ is the minimum value such that $\sigma$ nihilates, i.e., 
	$\sigma(\eps^*)=0$, we have that $\sigma(\epsilon) > 0$ as far as  $\eps < \eps^*$. 
Then, in order to decrease $\sigma$ towards $0$,
the ODE  \eqref{eq:odefree}
  (which is not norm-preserving)  has necessarily to increase the norm of
 the perturbation  by its gradient system structure.  The monotonicity of this  integration  comes from the choice of the integration scheme (look at Algorithm \ref{alg:controlODE} without the normalization step $3$) since we only  accept the computed solutions corresponding to decreasing values for the function to be minimized $\sigma$.  
We remark that this fact holds true till the norm of the perturbation is less than $\epsilon^*$ in order to avoid jumps on different branches of the same function. 
\begin{obs}
	The coupling of \eqref{eq:ode} and \eqref{eq:odefree} described in this section imply monotonicity 
	of $\sigma(\eps \delta(\eps))$ with respect to $\eps$ when  when
	$\eps < \eps^*$.
   This is because of the gradient system structure
	of both systems of ODEs and the choice of the integration schemes.
	\end{obs}

The numerical scheme for the integration of \eqref{eq:odefree} is very similar to the one presented in Algorithm \ref{alg:controlODE}, (with the obvious changes).
The difference with respect to the previous case is that, by removing the constraints on the computed solution during the integration of the equation, the norm of the perturbation $\delta$ increases, so that it is natural to link  the stopping criterion with the norm of such a perturbation. 
% \vspace{.5cm}

In this way we have a globally decreasing trajectory for $\sigma_{\min}(H_m(\pt))$ (with respect to the outer iteration), as we show in Algorithm \ref{alg:computeeps} 
(where the uppercase $f$ denotes that the corresponding quantities are associated with the free dynamic). 

\begin{algorithm2e}[H]
\caption{Computation of $\epsilon^*$}
\label{alg:computeeps}
\DontPrintSemicolon
\KwData{$p$, $tol$, $\Delta$}
\KwResult{$\epsilon_o$ approximation of $\epsilon^*$}

\Begin{
\nl Set $\epsilon=10^{-2}$           \ \ \ \ \ \ \ \ \ \ \ \ \      \%  starting value \;
\nl Integrate the equation \eqref{eq:ode} \;
 store $\delta, \sigma=\sigma_{\min} (H_m(p+ \epsilon \delta))$ \;
   \nl \While{$\sigma > tol$}
       {\nl $\epsilon_1=\epsilon + \Delta$  \;
    \nl       integrate the equation \eqref{eq:odefree} with initial value $p + \epsilon / \epsilon_1 \delta$ \;
           and stop when $ \epsilon / \epsilon_1 \| \delta^f \|_2 \geq 1$ \;
           store $\delta^{f}, \sigma^f=\sigma_{\min} (H_m(p+ \epsilon / \epsilon_1  \delta^f))$  \;
           set $\epsilon_1= \epsilon  \| \delta^f \|_2$   \;
       \nl  integrate equation \eqref{eq:ode} with initial value $p + \epsilon_1 \delta^f$  \;
         store $\delta, \sigma=\sigma_{\min}(H_m(p + \epsilon_1 \delta))$    \;
          set $\epsilon=\epsilon_1$  }

    \nl $\epsilon_o=\epsilon$

 }
\end{algorithm2e}

If we apply Algorithm \ref{alg:computeeps} to some randomly chosen test problem, we get the result of Figure~\ref{fig:testex}, where we can observe both the monotonicity of the objective functional and 
%the improvement in the value of the computed distance. The picture on the left in Figure~\ref{fig:testex} shows how the singular value decreases for an increasing value of $\epsilon$, while  the picture on the right of Figure~\ref{fig:testex} shows 
 the succession between free and constrained dynamics.
\begin{figure}[htb]
	\centering
	\includegraphics[width=.72\linewidth]{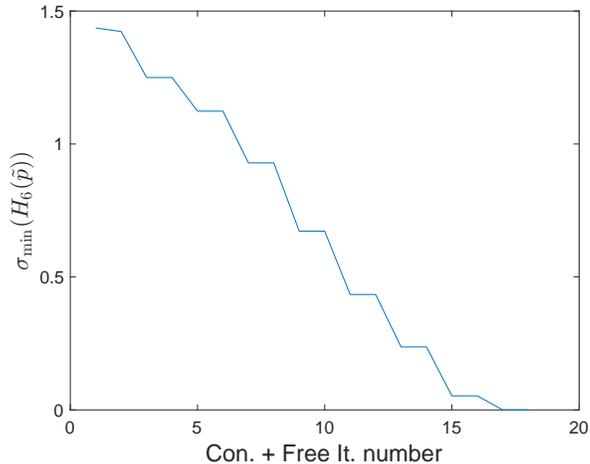}
	\caption{Plot of the functional $\sigma_{\min}(H_3(\hat{p}))$  as function of the  iteration number.\label{fig:testex}}
\end{figure}

\begin{obs}
In the previous sections we considered real numbers just to simplify the notation. All the arguments still hold true also in the complex case (as needed in the solution of the problem described in Section \ref{sec:polfrommoments}).
\end{obs}

\subsection{Different weights and missing coefficients}
A key result of the proposed approach is Lemma \ref{lemma:deriveig}, hence the algorithm is based on the solution of matrix nearness problems in the Frobenius metric. This motivates the choice of the Frobenius norm of the Hankel matrix in \eqref{eq:probFnorm}. The algorithm is then restated on the vector $p$ by  choosing $p = \ve (H_m(p))$.  However it is possible to add some different weights on the coefficients of $p$ by preserving all the properties previously shown.  This is done by appropriately changing the projection operator $P_{\mathcal{H}}$.  In particular,  once we average along the anti diagonals to compute the coefficients of the projected Hankel matrix, we multiply each entry of the output vector   by the corresponding weight. Using Matlab notation for the entry-wise product and denoting as $\delta_w$ the (normalized) vector whose entries are weighted, equation~\eqref{eq:ode} would be replaced by
$$
  \dot{\delta}_w = {}-w.* \ve(P_{\mathcal{H}} (uv^{\top})) + \langle \delta_w, w.*\ve(P_{\mathcal{H}} (uv^{\top})) \rangle \delta_w.
$$
Similarly we can also fix some entries of the vector $p$:  It is sufficient to set as $0$ the corresponding weights so that 
 the algorithm leaves the coefficients untouched during the iterations, and perturb only a subset of the entries of $p$.
 
In the case some coefficients are missing, it is enough to  choose some initial values for them (e.g. by averaging their neighbors) and run the algorithm.

\section{Numerical results}
\label{sec:numtest}
In this section we run some numerical examples.
We consider both problems described in Section \ref{sec:appl}.
In the problem of system identification we make a comparison with the  function \textit{slra} from the SLRA Toolbox \cite{UsevichMarkovsky2} in order to test the performances of the proposed algorithm  
%\footnote{the function \textit{slra.m} minimizes a weighted $2$ norm} 
 (we cannot do the same with the problem of polygons reconstruction since the available function \textit{slra} does not work with complex numbers).

\subsection{Identification of linear time-invariant  models}
The problem is the one described in Section \ref{sec:sysid}. 
The goal of the experiments is to recover a time series from its noisy entries. First, we construct a time series, which is a response of a linear time-invariant model. 
The simulation setup is as follows:
\begin{enumerate}
\item a random linear time-invariant system of order $n$ is selected (\texttt{drss} function in MATLAB);
\item the ``true data'' $p_0$ is the response of the system selected on step 1 to a random initial condition;
\item the ``noisy data'' $p$ is 
\begin{equation}
\label{eq:buildp}
p = p_0 + \tau r \|p_0\|_2/\|r\|_2,
\end{equation}
where $r$ is a zero mean white Gaussian random vector with unit variance and $\tau>0$ is parameter (called the noise level). 
\end{enumerate}

%As first example we choose the (exact)  time series  $p_0^1=[1, 2, \dots, 10]$, which  is the set of parameters we want to estimate. Indeed we observe the following relation holds
%\begin{equation}
%\label{ex:simpletimeseries}
%\begin{pmatrix}
%1 & -2 & 1
%\end{pmatrix} \begin{pmatrix}
%1 & 2 & 3 & \cdots \\
%2 & 3 & 4 & \cdots \\
%3 & 4 & 5 & \cdots
%\end{pmatrix}  = \begin{pmatrix}
%0 & 0 & 0 & \cdots
%\end{pmatrix}
%\end{equation}
%so that the system has order $2$.
%As second example we consider a more realistic (and complicated) signal. The paramater vector $p_0^2$ is a multisine wave coming from a linear time-invariant system of order $n$ and it consists of $30$ points (the number of parameters). 

In the following, we compare the results computed by the proposed approach with the ones computed by the function \textit{slra} from \cite{UsevichMarkovsky2}.
However the solution computed by the proposed method is not exactly rank deficient, so that a further comparison is done with a refinement obtained by \textit{combining} the two functions:  we initialize the function \textit{slra} with the solution found by the ODE based method. In this way, we obtain a rank deficient Hankel matrix. 

\paragraph{Test on the accuracy of the computed solution}
In this first example we test the accuracy on the computed solution.   On a single run  we expect that different local optimization methods  compute different solutions, hence we look at the average behavior on several runs of the two algorithms. 

Figure~\ref{fig:distsvdsord5} shows the results from a model of order $n = 5$.
\begin{figure}[htb]
	\centering
	\includegraphics[width=.72\linewidth]{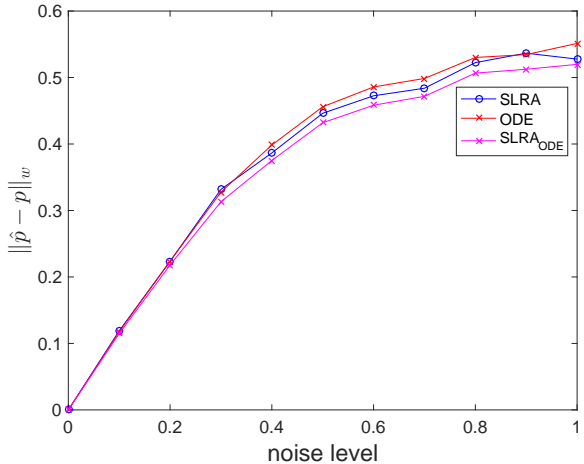}
	\caption{Identification of the output of a linear time-invariant model of order $5$ as function of the noise level  (average over fifty perturbations).\label{fig:distsvdsord5}}
\end{figure}
We observe how the distances associated with  all the computed solutions are very close; however the function \textit{slra} initialized with the solution of the proposed approach is able to  achieve smaller values of distance. The computational time of the function \textit{slra} is smaller than the one of the proposed ODE based method. 

The next experiment shows an interesting property of the proposed local optimization approach.  
\paragraph{Dependence on the initial estimate}
We analyzed in the previous experiment which is the distance computed by the different algorithms; however the solutions found by local optimization approaches usually  strongly depend on the initial estimate. An interesting advantage of the proposed algorithm appears to be the robustness with respect to the initial approximation, i.e., the computed solutions are (almost) independent from the initial estimate. In practice, this means that choosing two different initial directions, the proposed ODE-based algorithm finds (almost) the same solutions. 

The standard initialization for the two algorithms (the ones used by default) are the following:
\begin{itemize}
\item ODE: the starting optimal perturbation $\delta$ is chosen as the steepest descent direction for the smallest singular value of the  Hankel matrix $H_m(p)$, i.e., $\ve(-P_{\mathcal{H}}(uv^{\top}))$ (divided by its  norm), where $u, v$ are the left and right singular vector associated with $\sigma_{\min}(H_m(p))$;
\item SLRA: the initial approximation is the unstructured low-rank approximation of the Hankel matrix $H_m(p)$.
\end{itemize}
In the next experiment (Figure~\ref{fig:iniest}) we perturb (using a random perturbation coming from a normal distribution with zero mean and standard deviation $0.5$) the standard initial estimates of the two algorithms in order to observe how the final computed solutions change. (Again, we show the average behavior over several runs.)
	\begin{figure}[htb]
	\centering
	\includegraphics[width=.72\linewidth]{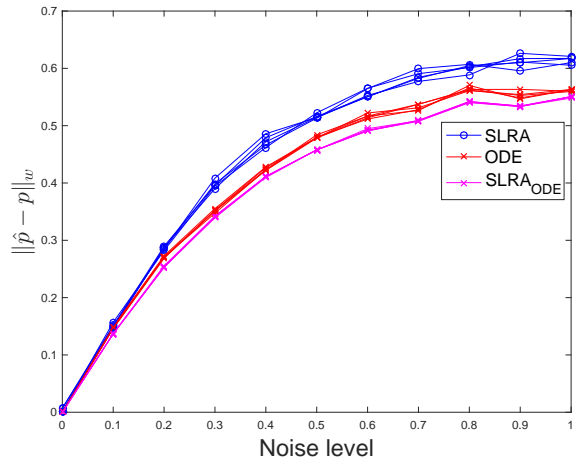}
	\caption{Identification of the output of  system of order $5$: error computed  as function of the noise level (average over fifty perturbations). Each plot corresponds to a different initial estimate. \label{fig:iniest}}
\end{figure}
It happens that all  the plots associated with the ODE based algorithm are very close (they are all almost overlapping) so we cannot distinguish them in the figure. On the other hand, about the function \textit{slra}, the different plots associated with different   initial estimate spread off more, especially when the level of noise increases (hence the optimization problem becomes more difficult). We can also observe that by changing the initial estimate the solution computed by \textit{slra} is less accurate (comparing with the results in Figure~\ref{fig:distsvdsord5}).   But the initialization of \textit{slra} with the solution computed by the ODE method still leads to the best results; in this case we can not distinguish the plots hence the independence from the initial condition is more clear (this is because the outputs of the ODE method are all very close). 

\begin{obs}
Because of normalization issues we are not able to establish a perfect link between the initial estimates for the two different algorithms. However  a perturbation on the optimal initial estimate (the one used by the algorithms by default) it should be enough to show the robustness of the proposed algorithm with respect to the initial condition.  
\end{obs}

\subsection{Reconstructing polygons from moments}
We show here some numerical examples for the problem described in Section \ref{sec:polfrommoments} or in details in \cite{MilanfarVerghese}. We consider a triangle whose vertices are given by the following three points in the complex plane:
\begin{equation}
\label{ex:vertices}
\begin{aligned}
 z_1&= -0.4655 + 0.2201 i, \\
 z_2&= \ \ 0.0082 + 0.4599 i \\
 z_3&= -0.3283 - 0.1809 i.
 \end{aligned}
\end{equation}
Following the procedure described in \cite{MilanfarVerghese} we aim at reconstructing such a triangle from a set of (perturbed) complex moments by solving a Hankel low-rank approximation problem. Since we deal with complex valued data  we only use the proposed method to study the problem without any comparison.
\begin{figure}[!ht] 
	\begin{minipage}[b]{0.5\linewidth}
		\label{fig:noise1}
		\centering
		\includegraphics[width=5cm, height=5cm]{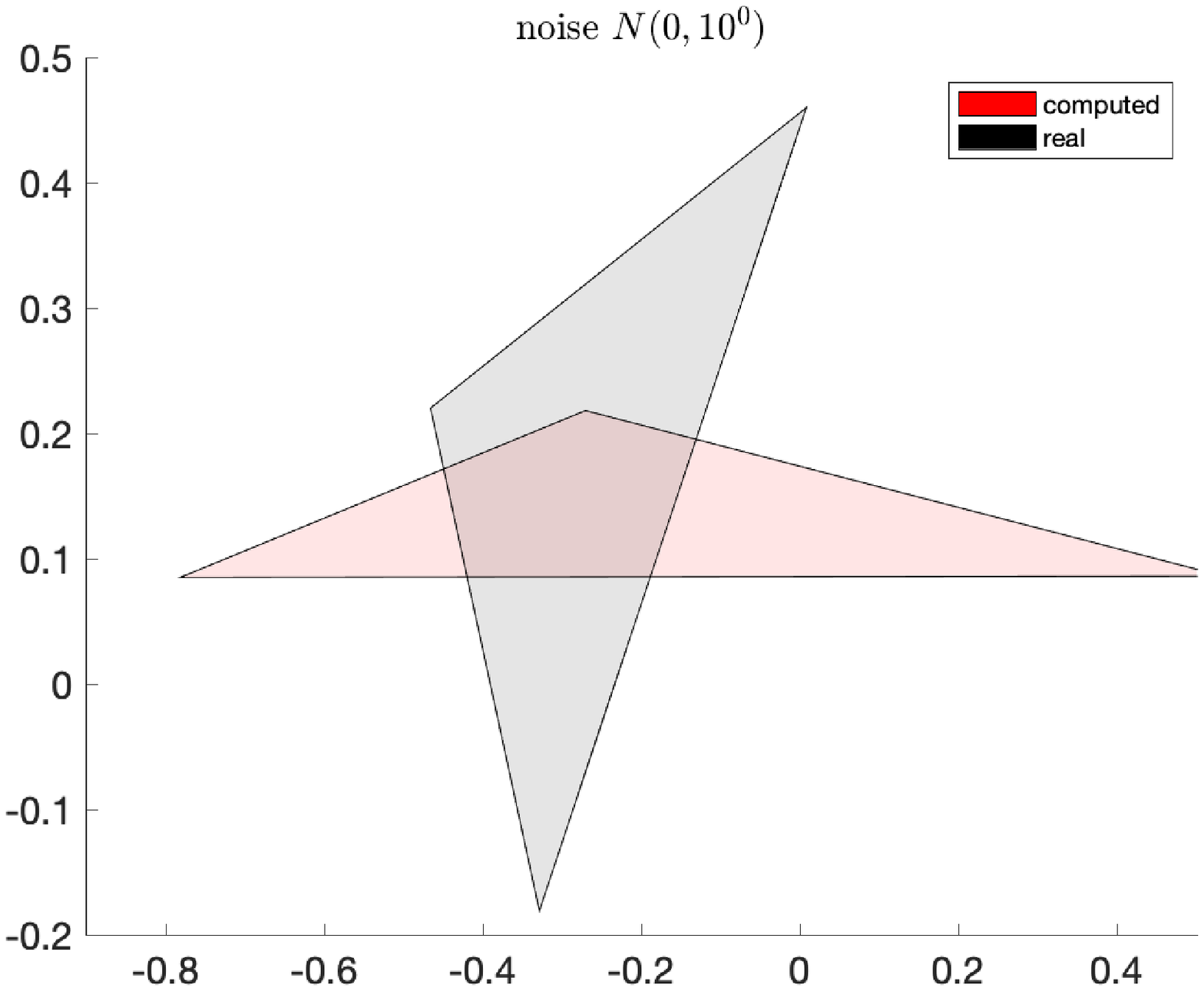} 
		\caption{Noise level $10^{0}$} 
		\vspace{4ex}
	\end{minipage}%%
	\begin{minipage}[b]{0.5\linewidth}
		\label{fig:noise2}
		\centering
		\includegraphics[width=5cm, height=5cm]{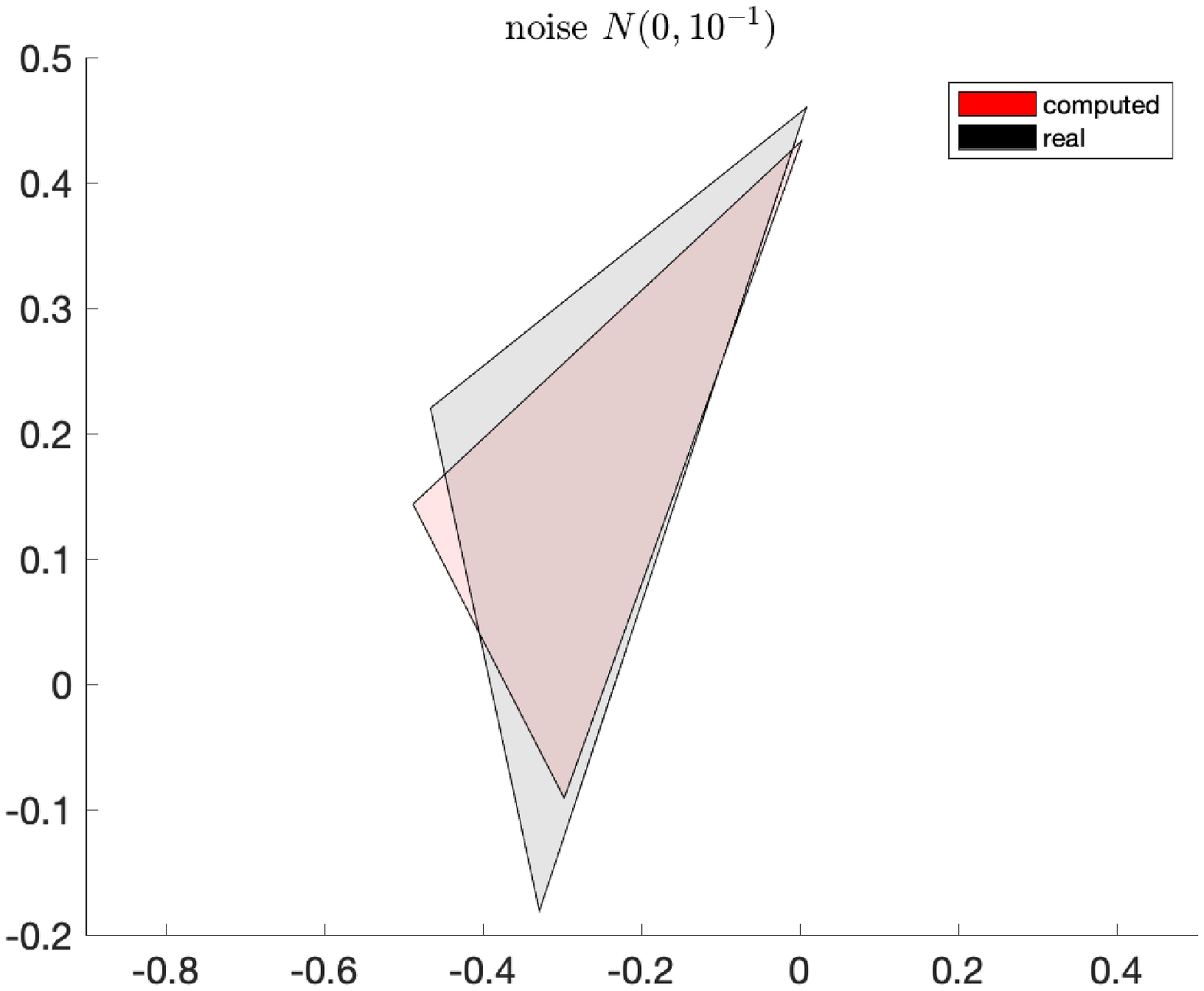} 
		\caption{Noise level $10^{-1}$} 
		\vspace{4ex}
	\end{minipage} 
	\begin{minipage}[b]{0.5\linewidth}
		\centering
		\includegraphics[width=5cm, height=5cm]{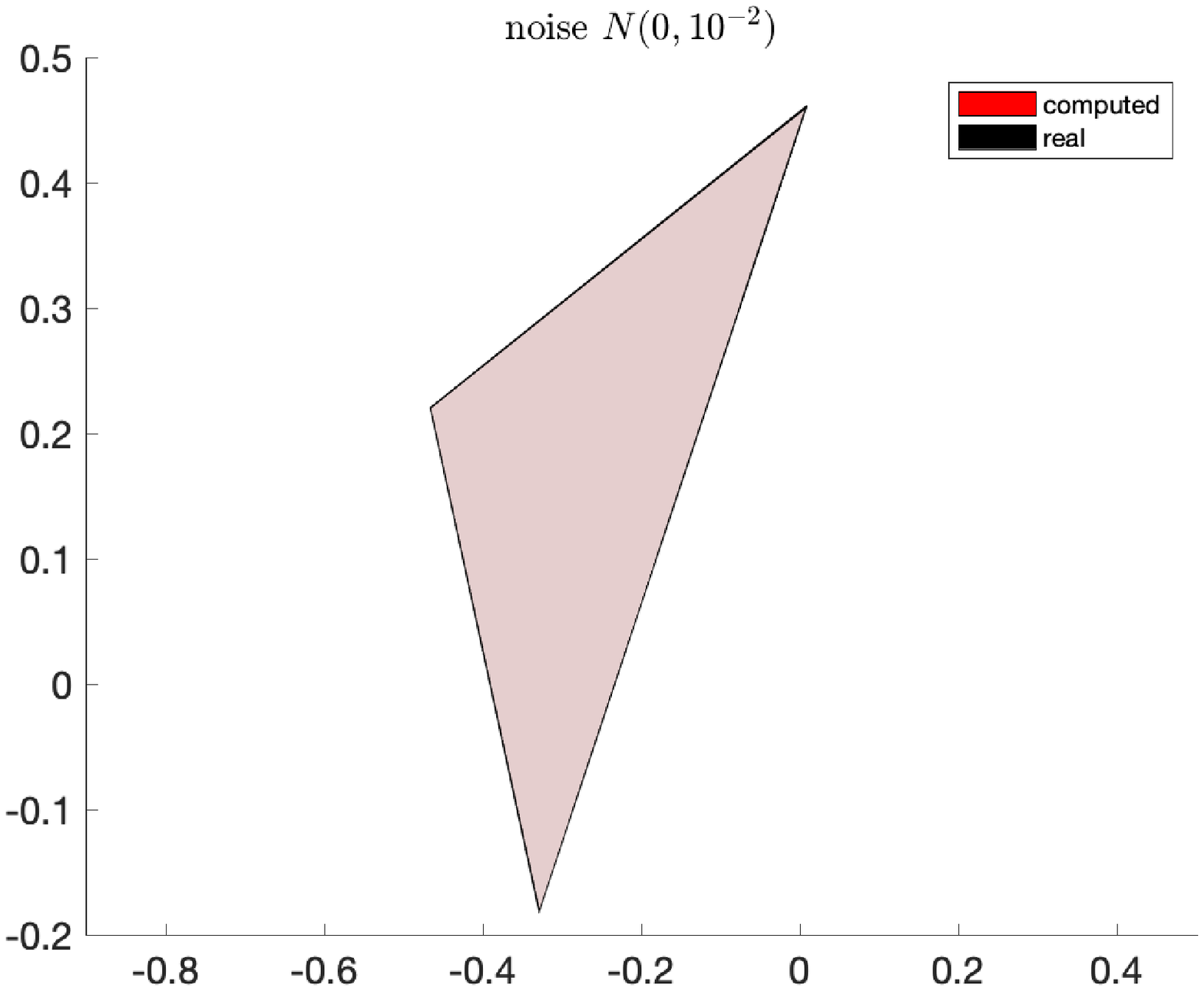} 
		\caption{Noise level $10^{-2}$, \label{fig:noise3}} 
		\vspace{4ex}
	\end{minipage}%% 
	\begin{minipage}[b]{0.5\linewidth}
		\label{fig:noise4}
		\centering
		\includegraphics[width=5cm, height=5cm]{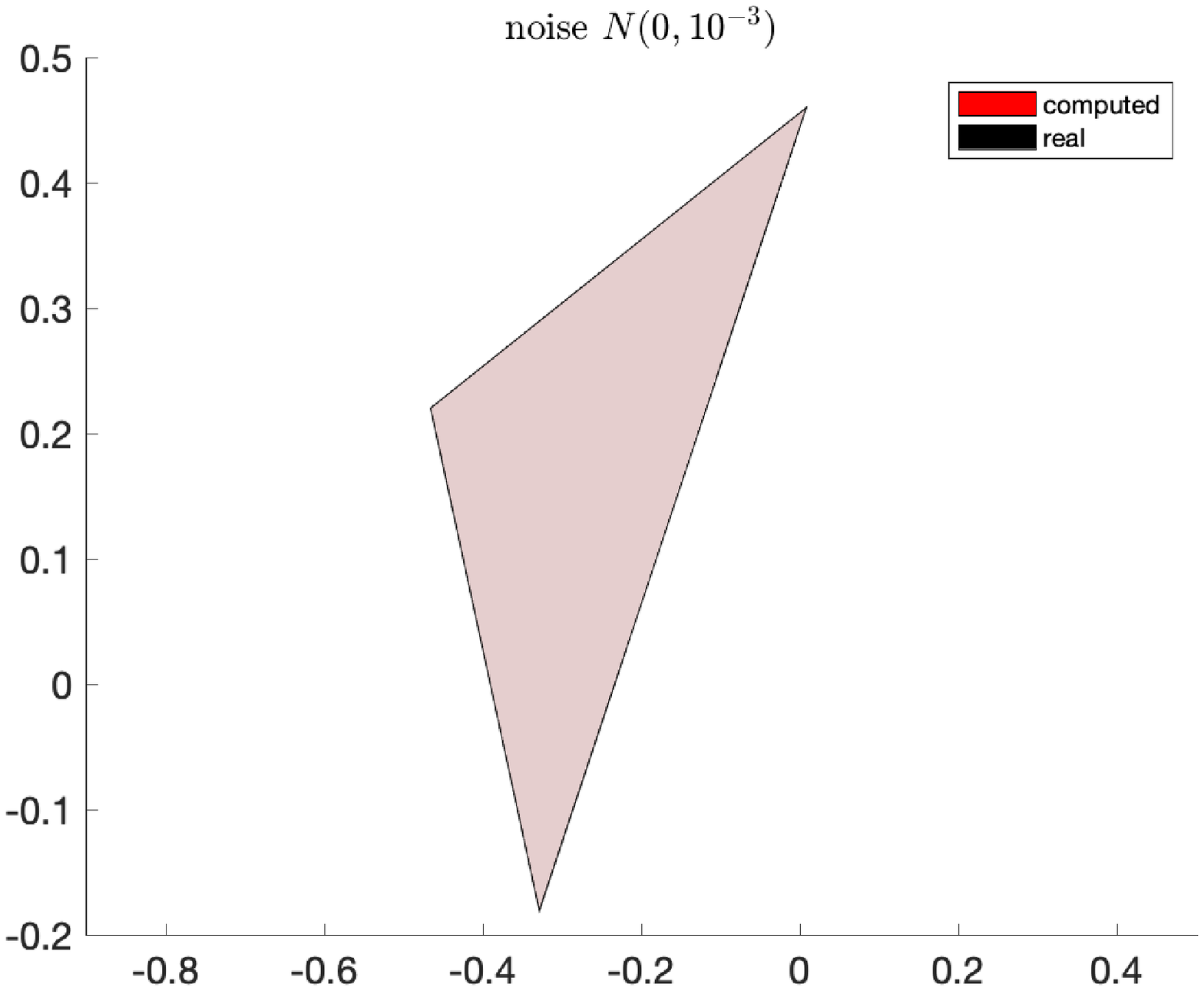} 
		\caption{Noise level $10^{-3}$} 
		\vspace{4ex}
	\end{minipage} 
	\caption{Triangle from moments reconstruction: numerical results for different noise levels \label{fig:graphic}}
\end{figure}

We start from a set $T = (\tau_0, \dots, \tau_{N-1})$ of $N$ complex moments and we add a random Gaussian perturbation to both their real and imaginary part. Then we solve the Hankel low-rank approximation problem for the matrix $H_{N-n}(T)$ and we recover the set of vertices from its kernel,  as described in Section \ref{sec:polfrommoments} \cite{MilanfarVerghese}. The error between the exact and the approximating solution is measured by looking at the two sets of vertices:
$$
e = \Bigg\lVert \begin{pmatrix}
z_1 \\ z_2 \\ z_3
\end{pmatrix} - \begin{pmatrix}
\hat{z}_1 \\ \hat{z}_2 \\ \hat{z}_3
\end{pmatrix} \Bigg\rVert_2
$$
where the vertices are ordered by decreasing real part.
 All the following results are the average over fifty runs (we generate fifty random perturbations and we consider the average solution so that the results of the experiments are not misleading), 
 
 First we show some graphical results, in order to see what actually happens. In Figure~\ref{fig:graphic}  we observe the numerical results for different levels of noise (the scalar $\tau$ in \eqref{eq:buildp}) and $9$  complex moments.

 In the following analysis we want to analyze how the error behaves as function of the level of noise (for a fixed number of moments) and as function of the number of moments (for a fixed level of noise).
 
 In Figure~\ref{fig:errorvsnoise}  we observe how the error increases linearly (using a logarithmic scale on both axis) with the level of noise (we considered~$9$ complex moments).
 \begin{figure}[htb]
 	\centering
 	\includegraphics[width=.72\linewidth]{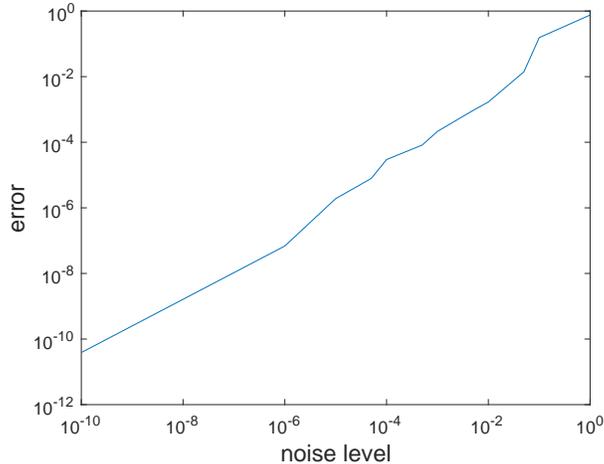}
 	\caption{\label{fig:errorvsnoise}Error as function of the noise level}
 \end{figure}
 
 The next analysis can help to understand how the error changes with the number of considered complex moments. We fix the noise level to $10^{-3}$. First of all we observe how the modulus of the (exact) complex moments behaves for an increasing number $N$ (Figure~\ref{fig:abst}).
 \begin{figure}[htb]
 	\centering
 	\includegraphics[width=.72\linewidth]{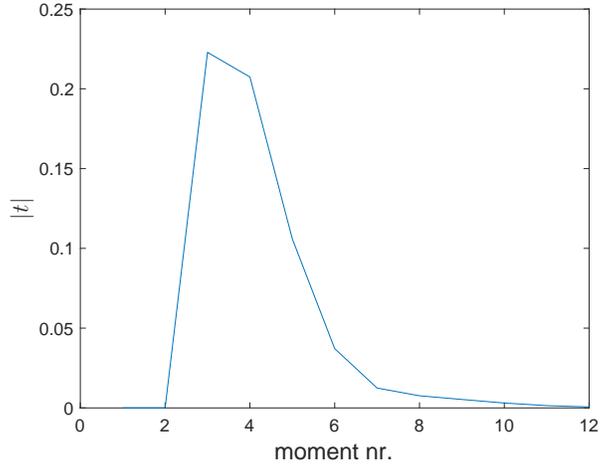}
 	\caption{\label{fig:abst}Modulus of the first (exact) $16$ complex moments for the vertices in~\eqref{ex:vertices}}
 \end{figure}
 The first two moments are always zero by definition. In \cite{MilanfarVerghese} it  is proved that a minimum of $2n+1$ moments ($n$ is the number of vertices) is necessary for the reconstruction. From Figure~\ref{fig:abst} we can see that
 the modulus of the moments is decreasing, so we can expect that the smaller is the magnitude the less is the influence of a moment in the reconstruction. 
  We can plot now how the error changes for an increasing number of complex moments (Figure~\ref{fig:errorvsmoments}), where we observe how the error is  decreasing for an increasing number of complex moments. Consequently, according  to the needed accuracy on the computed solution we can choose to use a certain  number of complex moments in order to optimize the computational cost.    
  \begin{figure}[!ht]
  	\centering
  	\includegraphics[width=.72\linewidth]{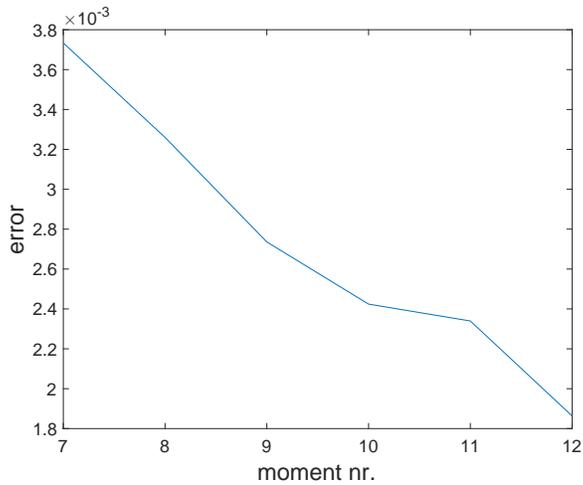}
  	\caption{\label{fig:errorvsmoments}Error for an increasing number of complex moments. The noise level is fixed.}
  \end{figure}
 
 \begin{obs}
 This last analysis on the dependence of the error from the number of complex moments is strongly linked to the choice of the set of vertices. This is because the complex moments (and their magnitudes) are functions of the vertices. 
 \end{obs}
 
 \section{Conclusion and future work}
 \label{sec:concl}
 We proposed a new algorithm for solving Hankel structured low-rank approximation problems. It is based on a double iteration method which makes the smallest singular value of the data matrix decreasing along the direction described by a gradient system. We saw how the algorithm performs similarly  as the function \textit{slra}  for what concerns the accuracy on the computed solution and how to use it as initial approximation in order to achieve an improvement. Moreover the proposed approach  turns out to be more robust with respect to the initial approximation given in input. 
 The algorithm can be extended to block Hankel matrices, and to mosaic Hankel matrices (block matrices whose blocks are Hankel matrices) arising in applications in the field of system theory and identification. A more challenging task is to extend the proposed  approach to the problem of Hankel low-rank approximation with multiple rank constraints appearing, e.g., in the common dynamics estimation problem in multi-channel signal processing \cite{Markovskycommondyn}.
 Moreover, a similar strategy can be adopted to compute rank reductions greater than one. 
 The optimization of the proposed algorithm in terms of computational cost and time can be object of future work. 

\section*{Acknowledgements}

The research leading to these results has received funding from the European Research Council (ERC) under the European Union's Seventh Framework Programme (FP7/2007--2013) / ERC Grant agreement number 258581 ``Structured low-rank approximation: Theory, algorithms, and applications'' and Fond for Scientific Research Vlaanderen (FWO) projects G028015N ``Decoupling multivariate polynomials in nonlinear system identification'' and G090117N ``Block-oriented nonlinear identification using Volterra series'';  and Fonds de la Recherche Scientifique (FNRS) -- FWO Vlaanderen under Excellence of Science (EOS) Project no 30468160 ``Structured low-rank matrix / tensor approximation: numerical optimization-based algorithms and applications''.

The authors thank K. Usevich for some useful comments and suggestions. 
\bibliographystyle{elsarticle-num}
\bibliography{ref.bib}

\end{document}